\documentclass[a4paper]{scrartcl}
\usepackage{amssymb}
\usepackage{amsmath}
\usepackage{xypic}
\usepackage{graphicx}

\newenvironment{proof}{\textit{Proof.} }{ ~ \hfill $\Box$\\\smallskip}

\newtheorem{theorem}{Theorem}[section]
\newtheorem{lemma}[theorem]{Lemma}

\newtheorem{cor}[theorem]{Corollary}

\newtheorem{definition}[theorem]{Definition}

\numberwithin{equation}{section}
\newcommand{\Q}{\mathbb{Q}}

\newcommand{\R}{\mathbb{R}}

\DeclareMathOperator{\image}{im}

\DeclareMathOperator{\Diff}{Diff}
\DeclareMathOperator{\Id}{Id}
\DeclareMathOperator{\Iso}{Iso}
\DeclareMathOperator{\aut}{Aut}
\DeclareMathOperator{\scal}{scal}

\title{On moduli spaces of positive scalar curvature metrics on highly connected manifolds}

\author{Michael Wiemeler\thanks{The research for this work was supported by DFG-Grant HA 3160/6-1, and by SFB 878 ``Groups, Geometry and Actions'' and the Cluster of Excellence ``Mathematics M\"unster'' at WWU M\"unster.}}

\date{ }

\begin{document}

\maketitle

\begin{abstract}
  Let \(M\) be a simply connected spin manifold of dimension at least six which admits a metric of positive scalar curvature.
  We show that the observer moduli space of positive scalar curvature metrics on \(M\) has non-trivial higher homotopy groups.

  Moreover, denote by \(\mathcal{M}_0^+(M)\) the moduli space of positive scalar cuvature metrics on \(M\) associated to the group of orientation-preserving diffeomorphisms of \(M\).
  We show that if \(M\) belongs to a certain class of manifolds which includes \((2n-2)\)-connected \((4n-2)\)-dimensional manifolds, then the fundamental group of \(\mathcal{M}_0^+(M)\) is non-trivial.
\end{abstract}


\section{Introduction}
\label{sec:intro}

The space \(\mathcal{R}(M)\)  of Riemannian metrics  on a closed orientable manifold \(M\) is an open \(\Diff(M)\)-invariant convex cone in the infinite dimensional vector space \(\Gamma(S^2(T^*M))\)  of symmetric \((2,0)\)-tensors on \(M\).
Here \(\Gamma(S^2(T^*M))\) is equipped with the \(C^\infty\)-topology.
\(\mathcal{R}(M)\) contains as an open \(\Diff(M)\)-invariant subset the space \(\mathcal{R}^+(M)\) of positive scalar curvature metrics on \(M\).
If \(G\subset \Diff(M)\) is a subgroup, we say that the orbit space \(\mathcal{M}_G^+(M)=\mathcal{R}^+(M)/G\) is the \emph{moduli space} of positive scalar curvature metrics on \(M\) associated to \(G\).
Of particular interest are here the cases where \(G=\Diff(M)\) is the whole diffeomeorphism group, \(G=\Diff_0(M)\) is the group of orientation-preserving diffeomorphisms, or \(G=\Diff_{ob}(M)\) is the observer diffeomorphism group, i.e., the group of all diffeomorphisms which fix a given point \(x_0\in M\) and act as the identity on \(T_{x_0}M\).
In these cases we let \(\mathcal{M}^+(M)=\mathcal{M}_{\Diff(M)}^+(M)\), \(\mathcal{M}_0^+(M)=\mathcal{M}_{\Diff_0(M)}^+(M)\) and \(\mathcal{M}^+_{ob}(M)=\mathcal{M}^+_{\Diff_{ob}(M)}(M)\).

In recent years a lot of work was devoted to the study of these spaces and moduli spaces.
Here we do not want to repeat the whole development of the subject.
Therefore we refer the reader to the book \cite{MR3445334} for an overview and to \cite{botvinnik14:_infin}, \cite{MR3073935}, \cite{crowley18:_harmon_gromol_toda}, \cite{MR3270591} and \cite{perlmutter17:_param_morse_theor_posit_scalar_curvat} for more recent developments.

In the two recent publications \cite{MR3270591} and \cite{botvinnik14:_infin} it is shown that, for a high dimensional spin manifold \(M\) which admits a metric of positive scalar curvature, the space \(\mathcal{R}^+(M)\) has non-trivial higher rational homotopy groups \(\pi_k(\mathcal{R}^+(M))\otimes \mathbb{Q}\) in certain degrees \(k\).

While for the construction in \cite{MR3270591} the degree \(k\) of these homotopy groups is bounded above by some constant depending on and growing with the dimension of the manifold, the proof in \cite{botvinnik14:_infin} does not require such a bound.
However, in \cite{MR3270591} it was also shown that the natural map
  \begin{equation*}
    \pi_k(\mathcal{R}^+(M),g_0)\otimes \Q\rightarrow\pi_k(\mathcal{M}^+_{ob}(M),[g_0])\otimes \Q
  \end{equation*}
  has non-trivial image if \(M\) is a \(\hat{A}\)-multiplicative fiber in degree \(k\) (see Definition \ref{sec:the-proof-theorem-2}).

  Here we observe that \(M\) is always a \(\hat{A}\)-multiplicative fiber in degree \(k\) if \(k>2\dim M\).
  Combing this observation with the results cited above leads to the following theorem:

\begin{theorem}
\label{sec:introduction}
  Let \(M\) be a closed spin manifold of dimension \(n\geq 6\) which admits a metric of positive scalar curvature.
  If \(k=4s -n-1 >2n\), \(s\in \mathbb{Z}\) and \(g_0\in \mathcal{R}^+(M)\), then 
  \begin{equation*}
    \pi_k(\mathcal{R}^+(M),g_0)\otimes \Q\rightarrow\pi_k(\mathcal{M}^+_{ob}(M),[g_0])\otimes \Q
  \end{equation*}
 has non-trivial image.

  If moreover, there is no non-trivial orientation-preserving action of a finite group on \(M\), then
  \begin{equation*}
    \pi_k(\mathcal{R}^+(M),g_0)\otimes \Q\rightarrow\pi_k(\mathcal{M}_0^+(M),[g_0])\otimes \Q
  \end{equation*}
has non-trivial image.
\end{theorem}

Examples of manifolds on which no finite group acts non-trivially and orientation-preserving have been given by Puppe \cite{MR1368664}.
These examples are simply connected spin manifolds of dimension six.
Therefore, by work of Gromov--Lawson \cite{MR577131}, Schoen--Yau \cite{MR535700} and Stolz \cite{MR1189863}, they admit metrics of positive scalar curvature.
Until now it is an open problem whether there is a simply connected manifold \(M\) which does not admit any non-trivial action of a finite group.
Using results of \cite{MR2680210}, the higher homotopy groups of \(\mathcal{M}^+(M)\) of such a manifold would be non-trivial if \(\mathcal{M}^+(M)\) is non-empty.
There are examples of manifolds \(M\) such that every compact Lie group which acts on one of them is finite.
For these examples the higher homotopy groups of \(\mathcal{M}^+(M)\) are also known to be non-trivial \cite{MR2680210}.

Regarding the fundamental group of spaces and moduli spaces of positive scalar curvature metrics we also show the following theorem:

\begin{theorem}
\label{sec:introduction-1}
Let \(n>1\) and \(M\) be a closed spin manifold of dimension \(4n-2\) which admits a metric of positive scalar curvature and is a \(\hat{A}\)-multiplicative fiber in degree \(1\).
Let, moreover, \(g_0\in\mathcal{R}^+(M)\).
Then the image of the map \[\pi_1(\mathcal{R}^+(M),g_0)\rightarrow \pi_1(\mathcal{M}_0^+(M),[g_0])\] contains elements of infinite order.
\end{theorem}

Examples of \(\hat{A}\)-multiplicative fibers in degree \(1\) are manifolds whose even degree rational cohomology vanishes in all degrees except degree \(0\) and the top degree (see Lemma~\ref{sec:hata-mult-fibers}).
In particular, this holds for spheres, or more generally for \((2n-2)\)-connected \((4n-2)\)-dimensional manifolds, or connected sums of several copies of products \(S^{2m+1}\times S^{2m'+1}\) of odd dimensional spheres.

For the proof of this theorem we combine the results of \cite{MR3270591} and \cite{botvinnik14:_infin} with a  result of Bourguignon \cite{MR0418147}, which allows to lift paths in \(\mathcal{M}(M)\) to paths in \(\mathcal{R}(M)\), and transversality considerations.
While the transversality considerations generalize to the situation of higher homotopy groups, the path lifting result does not.
Therefore we do not know whether Theorem~\ref{sec:introduction-1} generalizes to higher homotopy groups.

As far as we know, the classes constructed in Theorem~\ref{sec:introduction-1} are the first examples of elements in the fundamental group of the space of positive scalar curvature metrics which decent to elements of infinite order in the fundamental group of \(\mathcal{M}_0^+(M)\).
Note that, by \cite{MR3268776}, \(\pi_1(\mathcal{R}^+(S^{4n-2}))\) is abelian.
The space \(\mathcal{R}^+(S^2)\) is known to be contractible \cite{MR1818778}.
Hence, the above theorem is false for \(n=1\). 

Using the same techniques as in the proof of Theorem~\ref{sec:introduction-1} we can also prove the following theorem.

\begin{theorem}
  \label{sec:introduction-2}
  Let \(n>2\) be odd and \(0<m_1<m_2<n\) such that \(m_1+m_2=n\).
Moreover, let \(N\) be a closed connected orientable manifold admitting a metric of positive scalar curvature such that
\begin{enumerate}
\item \(H^k(N;\mathbb{Z})=0\) for \(k=m_2-m_1,m_1,m_2,n\),
\item all automorphisms of \(H^*(N;\mathbb{Z})\) are orientation preserving, i.e., induce multiplication with \(1\) in top degree.
\end{enumerate}

  Then, for \(M=S^{m_1}\times S^{m_2}\times S^n\times N\), 
  \[\pi_1(\mathcal{M}_0^+(M),[g_0])\]
  contains elements of infinite order.
  Here \(g_0=h+\epsilon^2 g\) is a product metric with \(h\in \mathcal{R}(S^{m_1}\times S^{m_2}\times S^n)\), \(g\in \mathcal{R}^+(N)\) and \(\epsilon>0\) sufficiently small.
\end{theorem}

Examples of manifolds for which the above assumptions on \(N\) hold are complex projective spaces \(\mathbb{C} P^{2k}\) of even complex dimension, quaternionic projective spaces \(\mathbb{H}P^{2k}\) of even quaternionic dimension and the Cayley plane.

It follows from the proof of Theorem~\ref{sec:introduction-2} that the elements in the fundamental group of \(\pi_1(\mathcal{M}_0^+(M),[g_0])\) constructed there are not in the image of the map \(\pi_1(\mathcal{R}^+(M),g_0)\rightarrow\pi_1(\mathcal{M}_0^+(M),[g_0])\).

This note is structured as follows.
 In the next Section \ref{sec:proof} we prove Theorem~\ref{sec:introduction}.
In Section~\ref{sec:trans}
 we recall some transversality results in the context of infinite dimensional manifolds.
Then in Section \ref{sec:ebin} we recall Ebin's slice theorem and some of its consequences.
In Section \ref{sec:proof2} we give the proofs of Theorems~\ref{sec:introduction-1} and \ref{sec:introduction-2}.

I would like to thank Anand Dessai for comments on an earlier version of this article.

\section{$\hat{A}$-multiplicative fibers}
\label{sec:proof}

In this section we recall the definition of $\hat{A}$-multiplicative fibers and prove Theorem~\ref{sec:introduction}.
Our proof of this theorem is based on the following result:

\begin{theorem}[{\cite[Theorem A]{botvinnik14:_infin}}]
\label{sec:the-proof-theorem}
  Let \(M\) be a closed spin manifold of dimension \(n\geq 6\).
  If \(k= 4s - n- 1\geq 0\) and \(g_0\in \mathcal{R}^+(M)\), then the map
  \begin{equation*}
    A_k\otimes \Q :\pi_k(\mathcal{R}^+(M),g_0)\otimes \Q\rightarrow \Q
  \end{equation*}
  is surjective.
\end{theorem}

Here \(A_k\) denotes the secondary index invariant for metrics of positive scalar curvature metrics.
There are two definitions for this invariant one is by Hitchin \cite{MR0358873}, the other is by Gromov and Lawson \cite{MR720933}.
However, as shown in \cite{ebert14}, these two definitions lead to the same invariant.

We also need the following definition from \cite{MR3270591}:

\begin{definition}
  \label{sec:the-proof-theorem-2}
  Let \(M\) be an oriented closed smooth manifold. We call \(M\) a \(\hat{A}\)-multiplicative fiber in degree \(k\) if for every oriented fiber bundle \(M\rightarrow E\rightarrow S^{k+1}\) we have \(\hat{A}(E)=0\).
\end{definition}

The following two lemmas give sufficient conditions for general manifolds to be \(\hat{A}\)-multiplicative in some degree.

\begin{lemma}
  \label{sec:the-proof-theorem-3}
  Let \(M\) be a closed oriented smooth manifold of dimension \(n\geq 1\).
  If \(k>2n\), then \(M\) is a \(\hat{A}\)-multiplicative fiber in degree \(k\).
\end{lemma}
\begin{proof}
  Let \(M\rightarrow E\rightarrow S^{k+1}\) be a smooth oriented fiber bundle.
  The tangent bundle of \(E\) is isomorphic to \(p^*(TS^{k+1})\oplus V\) where \(V\) is the bundle along the fiber.
  Since \(TS^{k+1}\) is stably trivial, it follows that \(TE\) and \(V\) are stably isomorphic.
  Therefore the Pontrjagin classes of \(E\) are concentrated in degrees smaller or equal to \(2n\).
  Moreover, since \(n<2n<k\), it follows from an inspection of the Serre spectral sequence for the fibration that \(H^{j}(E;\Q)=0\) for \(n<j<k+1\).
  Because \(2n<k\) the Pontrjagin classes of \(E\) are concentrated in degrees smaller or equal to \(n\).
  Moreover it follows that all products of these classes of degree greater than \(n\) must vanish.
  In particular all the Pontrjagin numbers of \(E\) vanish and the lemma is proved.
\end{proof}

\begin{lemma}
  \label{sec:hata-mult-fibers}
  Let \(M\) be a closed oriented smooth manifold of dimension \(n\geq 1\).
  Assume that the even degree rational cohomology of \(M\) vanishes in all degrees except degree \(0\) and the top degree.

  Then \(M\) is a \(\hat{A}\)-multiplicative fiber in degree one.
\end{lemma}
\begin{proof}
  Let \(M\rightarrow E\rightarrow S^2\) be a smooth oriented fiber bundle.
  Then it follows from an inspection of the Serre spectral sequence, that \(H^{4k}(E;\mathbb{Q})=0\) for all \(k>0\) with \(4k\neq \dim E=n+2\).
  Therefore all rational Pontrjagin classes of \(E\), except maybe \(p_{(n+2)/4}(E)\), vanish.
  However, \(p_{(n+2)/4}(E)\) also vanishes because the signature is multiplicative in fiber bundles with simply connected base \cite{MR0087943}.

  Hence it follows that \(\hat{A}(E)=0\) and the lemma is proved.
\end{proof}

For \(\hat{A}\)-multiplicative manifolds the following is known:

\begin{theorem}
\label{sec:the-proof-theorem-1}
  Let \(M\) be a closed spin manifold and \(k\geq 2\). If \(M\) is a \(\hat{A}\)-multiplicative fiber in degree \(k\), then the following holds:
  \begin{enumerate}
  \item The map \(A_k\otimes \Q\) from above factors through \(\pi_k(\mathcal{M}_{ob}^+(M),[g_0])\otimes \Q\).
  \item If there is no non-trivial smooth orientation-preserving action of a finite group on \(M\) then \(A_k\otimes \Q\) factors through \(\pi_k(\mathcal{M}_0^+(M),[g_0])\otimes \Q\).
  \end{enumerate}
\end{theorem}
\begin{proof}
  The first statement is proved in \cite[Section 2]{MR3270591}.
  The proof of the second statement is similar. Since there is no non-trivial orientation-preserving smooth action of a finite group on \(M\) the isometry group of any Riemannian metric on \(M\) contains at most one orientation-reversing involution and the identity.
  Therefore the group \(\Diff_0(M)\) of orientation-preserving diffeomorphisms acts freely on \(\mathcal{R}^+(M)\).

  In particular, there is an exact sequence
  \begin{equation*}
    \pi_k(\Diff_0(M),\Id_M)\otimes \Q\rightarrow \pi_k(\mathcal{R}^+(M),g_0)\otimes \Q\rightarrow \pi_k(\mathcal{M}_0^+(M),[g_0])\otimes \Q.
  \end{equation*}

  Since \(M\) is an \(\hat{A}\)-multiplicative fiber in degree \(k\) it follows from the arguments in \cite[Section 2]{MR3270591} that \(A_k\otimes \Q\) vanishes on the image of \(\pi_k(\Diff_0(M),\Id_M)\otimes \Q\).
  Therefore the theorem follows.
\end{proof}

Our Theorem~\ref{sec:introduction} now follows from the two theorems in this section and Lemma \ref{sec:the-proof-theorem-3}.

\section{Transversality}
\label{sec:trans}

In this section we start to collect the necessary material for the proofs of Theorems~\ref{sec:introduction-1} and \ref{sec:introduction-2}.
We need the following transversality result for maps into infinite dimensional vector spaces.

\begin{lemma}
\label{sec:transversality}
  Let \(F\) be a topological vector space and \(p:F\rightarrow \mathbb{R}^{n+1}\) a continuous linear surjective map and \(U\subset F\) an open neighborhood of zero.

If \(f:D^n\rightarrow F\) is a continuous map and \(K_1,\dots, K_m\subset D^n\) are compact such that there are open \(U_1,\dots,U_m\subset U\) with \(f(K_i)\subset U_i\),
then there is a continuous map \(f':D^n\rightarrow U\), such that
\begin{enumerate}
\item \(f'(D_{1/2}^n)\subset U-\{0\}\),
\item there is a homotopy \(h_t\)  from \(f\) to \(f'\), such that \(h_t(K_i)\subset U_i\) for all \(t\) and \(i\) and \(h_t|_{\partial D^n}=f|_{\partial D^n}\).
\end{enumerate}
\end{lemma}
\begin{proof}
  The existence of \(p\) guarantees that \(F\) is isomorphic as a topological vector space to \(\mathbb{R}^{n+1}\times F'\) for some closed subvector space \(F'\subset F\).
Therefore the lemma follows from the finite dimensional case \(F=\R^{n+1}\).
This case follows from transversality considerations as in the proof of Lemma 18.5 of \cite{MR0440554}.
\end{proof}

Note that,  if there is a continuous scalar product on an infinite dimensional topological vector space, then there is always a map \(p\) as above.
One can define \(p\) to be the orthogonal projection onto some \((n+1)\)-dimensional subvector space.
Note moreover, that if \(F\) is an infinite dimensional Banach space, then the existence of \(p\) is guaranteed by the Theorem of Hahn and Banach.
However, in the situation, in which we will apply the above lemma, \(F\) is not a Banach space; but has a continuous scalar product.

\section{Consequences of Ebin's slice theorem}
\label{sec:ebin}

Besides the transversality result from the previous section we need some knowledge about the local structure of the \(\Diff(M)\)-action on \(\mathcal{R}(M)\) to prove Theorems~\ref{sec:introduction-1} and \ref{sec:introduction-2}.
The basic structure result which we need is Ebin's slice theorem \cite{MR0267604} (see also \cite{MR0418147} and \cite{corro19:_short_rieman}).
Before we state it we fix some notations for the rest of this section.

We denote by \(\mathcal{R}\) an open \(\Diff(M)\)-invariant subset of the space of Riemannian metrics on the closed manifold \(M\).
We let \(G\subset \Diff(M)\) be a closed subgroup which contains the identity component of \(\Diff(M)\).
For example, we could have \(G=\Diff(M)\) or \(G=\Diff_0(M)\).
Moreover, we let \(\mathcal{M}=\mathcal{R}/G\).
For a metric \(g\in \mathcal{R}\) we denote by \(\Iso(g)\subset \Diff(M)\) the isometry group of \(g\) and by \(\Iso_G(g)\) the intersection of \(\Iso(g)\) with \(G\).
Note that, since \(M\) is compact, \(\Iso(g)\) and \(\Iso_G(g)\) are compact Lie groups.
Moreover, note that \(\mathcal{R}\) is a Fr\'echet manifold because it is an open subset of the Fr\'echet space \(\Gamma(S^2(T^*M))\) which is equipped with the \(C^{\infty}\)-topology.

Now we can state the slice theorem of Ebin.

\begin{theorem}
  For every \(g\in \mathcal{R}\) there exists a contractible \(\Iso(g)\)-invariant Fr\'echet submanifold \(S\) of \(\mathcal{R}\) containing \(g\) such that
  \begin{enumerate}
  \item If \(\eta\in \Diff(M)\) and \((\eta S)\cap S\neq \emptyset\), then \(\eta\in \Iso(g)\).
  \item There exists a local cross-section \(\chi:\Diff(M)/\Iso(g)\rightarrow \Diff(M)\) defined in a neighborhood \(U\) of the identity coset such that if \(\Phi:U\times S\rightarrow \mathcal{R}\) is defined as \((u,s)\mapsto \chi(u)s\), then \(\Phi\) is a homeomorphism onto a neighborhood of \(g\).
  \end{enumerate}
\end{theorem}

Since \(S\) is a manifold, there is a neighborhood of \(g\) in \(S\) which is \(\Iso(g)\)-equivariantly homeomorphic to an open invariant contractible subset \(D(E)\) of some topological vector space \(E\) with \(0\in D(E)\), where \(\Iso(g)\) acts linearly on \(E\).
Next we want to gather some properties of this vector space \(E\).
We follow \cite{MR0418147} for the description of these properties.

First note that on \(\Gamma(S^2(T^*M))\) we can define a continuous scalar product \(\langle\cdot,\cdot\rangle_g\)   by
\begin{equation*}
  \langle h,h'\rangle_g=\int_M \;g(h,h')\; dvol_g,
\end{equation*}
where \(h,h'\in \Gamma(S^2(T^*M))\) and \(dvol_g\) denotes the volume element of the Riemannian metric \(g\).

Moreover, the tangent space of the \(\Diff(M)\)-orbit of \(g\) in \(g\) is given by the image \(V\) of the linear operator
\begin{align*}
  \delta_g:\Gamma(TM)&\rightarrow \Gamma(S^2(T^*M))& \delta_g(X)&=\frac{1}{2}\mathcal{L}_Xg,
\end{align*}
where \(\mathcal{L}_X\) denotes the Lie derivative.
It has been shown by Ebin that \(V\) is a closed subspace of \(\Gamma(S^2(T^*M))\) and that it has an orthogonal complement with respect to \(\langle\cdot,\cdot\rangle_g\), i.e. there exists a closed subvector space \(W\) of \(\Gamma(S^2(T^*M))\), such that
\begin{enumerate}
\item \(\Gamma(S^2(T^*M))=V\times W\) as a topological vector space, and
\item \(\langle V,W\rangle_g=0\).
\end{enumerate}
The vector space \(E\) from above is this \(W\).

Note that by the \(\Iso(g)\)-invariance of \(V\) and the scalar product, \(E\) is \(\Iso(g)\)-invariant.
The next step is to understand the \(\Iso(g)\)-action on \(E\).

Note first that there is a projection map
\begin{align*}
  p:\Gamma(S^2(T^*M))&\rightarrow \Gamma(S^2(T^*M))^{\Iso(g)}&p(h)(v,w)&=\int_{\Iso(g)}(\gamma^*h)(v,w)\;d\gamma,
\end{align*}
where \(v,w\in TM\) and \(d\gamma\) denotes the Haar measure on \(\Iso(g)\).

This projection map restricts to a projection \(p':E\rightarrow E^{\Iso(g)}\).
Indeed, for \(X\in \Gamma(TM)\) and \(h\in E\) we have:
\begin{align*}
  \langle p(h), \mathcal{L}_Xg\rangle_g&=\int_M\int_{\Iso(g)} g(\gamma^*h,\mathcal{L}_Xg)\;d\gamma\;dvol_g\\
                                            &= \int_{\Iso(g)}\int_M g(h,\mathcal{L}_{\gamma^*X}g)\;dvol_g\;d\gamma\\
  &=\int_{\Iso(g)}\langle h,\mathcal{L}_{\gamma^*X}g\rangle_g\;d\gamma=0.
\end{align*}

Therefore there is a closed subvector space \(F\) of \(E\) such that \(E=E^{\Iso(g)}\times F\) as a topological vector space.
By \cite[Proposition III.20]{MR0418147}, \(F\) is infinite dimensional if \(\Iso(g)\) is non-trivial.

Note, that by the slice theorem \(\mathcal{M}\) is locally modeled on quotients \(D(E)/\Iso_G(g)\).
In the following we will call these local models \(D(E)/\Iso_G(g)\) Ebin charts of \(\mathcal{M}\).
Furthermore, we refer to \(D(E)\) as an Ebin slice.

Note, moreover, that \(\mathcal{R}\) and \(\mathcal{M}\) can be split into strata which consist of metrics with the same isometry group up to conjugation in \(\Diff(M)\).
As a consequence of the slice theorem, for every \(g'\in \Phi(U\times S)\), \(\Iso(g')\) is conjugated in \(\Diff(M)\) to a subgroup of \(\Iso(g)\) (see \cite[Theorem 8.1]{MR0267604}).
Furthermore, it follows from the proof of that theorem that the minimal stratum in \(\Phi(U\times S)\), i.e. the set of those metrics in \(\Phi(U\times S)\) whose isometry group is conjugated to \(\Iso(g)\), is given by \(\Phi(U\times S^{\Iso(g)})\).
Since the vector space \(F\) from above has infinite dimension, it follows from Lemma~\ref{sec:transversality} that maps from finite dimensional manifolds to \(\Phi(U\times S)\) can be made transverse to the minimal stratum.

For our argument we do not need the full knowledge about the stratifications of \(\mathcal{R}\) and \(\mathcal{M}\).
However, we have to study the following subspaces of \(\mathcal{M}\) and \(\mathcal{R}\) which are defined via the stratification.
For \(m,k\in \mathbb{N}\), \(m+k\geq 1\), we denote by \(\mathcal{M}_{m,k}\) (by \(\mathcal{R}_{m,k}\)) the subspace of \(\mathcal{M}\) (of \(\mathcal{R}\), respectively) which consists of classes of those metrics \([g]\) (of those metrics \(g\), respectively) with \(\dim \Iso(g)< m\) or (\(\dim \Iso(g)=m\) and \(|\Iso(g)/\Iso(g)^0|\leq k\)). Here \(\Iso(g)^0\) denotes the identity component of \(\Iso(g)\).

Note that \(\mathcal{R}_{m,k}\) is \(\Diff(M)\)-invariant and \(\mathcal{M}_{m,k}=\mathcal{R}_{m,k}/G\).
Moreover, \(\mathcal{R}_{0,1}\) is the dense open subset of \(\mathcal{R}\) which consists of those metrics which have no non-trivial isometries.
Therefore, \(\Diff(M)\) acts freely on \(\mathcal{R}_{0,1}\) and, by the slice theorem, the orbit map \(\mathcal{R}_{0,1}\rightarrow \mathcal{M}_{0,1}\) is a fibration.

The inclusions \(\mathcal{M}_{m,k-1}\subset \mathcal{M}_{m,k}\subset \mathcal{M}\) are inclusions of open sets because each \([g]\in \mathcal{M}_{m,k-1}\) has a neighborhood \(N\) such that for each \([g']\in N\), \(\Iso(g')\) is conjugated to a subgroup of \(\Iso(g)\).

Since \(S^n\) and \(D^{n+1}\) are compact for all \(n\in \mathbb{N}\), we have the following lemma about the homology and homotopy groups of the \(\mathcal{M}_{m,k}\) and \(\mathcal{R}_{m,k}\).

\begin{lemma}
  \label{sec:cons-ebins-slice-1}
  Let \(\mathcal{P}_{m,k}=\mathcal{M}_{m,k}\) or \(\mathcal{P}_{m,k}=\mathcal{R}_{m,k}\) for \(m,k\in\mathbb{N}\), \(m+k\geq 1\).
  Then for all \(n\in \mathbb{N}\) we have:
  \begin{enumerate}
\item For \(m\geq 1\): \(\pi_n(\mathcal{P}_{m,0})=\varinjlim_k \pi_n(\mathcal{P}_{m-1,k})\)
\item \(\pi_n(\mathcal{P})=\varinjlim_m \pi_n(\mathcal{P}_{m,0})\)
\item For \(m,k\in \mathbb{N}\), \(m+k\geq 1\): \(\ker(\pi_n(\mathcal{P}_{m,k})\rightarrow \pi_n(\mathcal{P}_{m+1,0}))= \varinjlim_j \ker(\pi_n(\mathcal{P}_{m,k})\rightarrow \pi_n(\mathcal{P}_{m,j}))\)
\item For \(m\geq 1\): \(\ker(\pi_n(\mathcal{P}_{m,0})\rightarrow \pi_n(\mathcal{P}))= \varinjlim_j \ker(\pi_n(\mathcal{P}_{m,0})\rightarrow \pi_n(\mathcal{P}_{j,0}))\)
\item\label{item:1} For \(m,k\in \mathbb{N}\), \(m+k\geq 1\): \(\ker(H_n(\mathcal{P}_{m,k})\rightarrow H_n(\mathcal{P}_{m+1,0}))= \varinjlim_j \ker(H_n(\mathcal{P}_{m,k})\rightarrow H_n(\mathcal{P}_{m,j}))\)
\item\label{item:2} For \(m\geq 1\): \(\ker(H_n(\mathcal{P}_{m,0})\rightarrow H_n(\mathcal{P}))= \varinjlim_j \ker(H_n(\mathcal{P}_{m,0})\rightarrow H_n(\mathcal{P}_{j,0}))\)
\end{enumerate}
  \end{lemma}
  \begin{proof}
    We only indicate the proof of the first claim.
    The proofs of the other claims are similar.
    Note that there is a natural map \(\Psi:\varinjlim_k\pi_n(\mathcal{P}_{m-1,k})\rightarrow \pi_n(\mathcal{P}_{m,0})\) induced by the inclusions \(\mathcal{P}_{m-1,k}\rightarrow \mathcal{P}_{m,0}\).

    We first show that \(\Psi\) is surjective.
    Let \(f:S^n\rightarrow \mathcal{P}_{m,0}\) be a map representing some element \([f]\in\pi_n(\mathcal{P}_{m,0})\).
    Then, because \(f(S^n)\) is compact and \(\mathcal{P}_{m,0}\) is the union of the open subsets \(\mathcal{P}_{m-1,k}\), there is a \(k>0\) such that \(f(S^n)\subset\mathcal{P}_{m-1,k}\).
    Hence, \([f]\in \image \Psi\).
    Therefore it follows that \(\Psi\) is surjective.
    Injectivity of \(\Psi\) follows similarly by using the compactness of \(D^{n+1}\).
    Therefore \(\Psi\) is an isomorphism.
  \end{proof}

For the spaces \(\mathcal{R}_{m,k}\) we also have the following stronger result.

\begin{lemma}
  \label{sec:cons-ebins-slice}
  For \(k\geq 1\), \(m+k\geq 2\), the inclusion \(\mathcal{R}_{m,k-1}\hookrightarrow \mathcal{R}_{m,k}\) is a weak homotopy equivalence, i.e. it induces an isomorphism on all homotopy groups.

  Therefore, the inclusion \(\mathcal{R}_{0,1}\hookrightarrow \mathcal{R}\) is a weak homotopy equivalence.
\end{lemma}
\begin{proof}
  As pointed out above, any map \(S^n\rightarrow \mathcal{R}_{m,k}\), \(n\in \mathbb{N}\), can be made transverse to the minimal stratum \(\mathcal{R}_{m,k}-\mathcal{R}_{m,k-1}\).
  Therefore it is homotopic to a map into \(\mathcal{R}_{m,k-1}\).
  Since a similar statement holds for maps defined on discs \(D^{n+1}\), the first claim follows.

  The second claim follows from the first and second claims of Lemma~\ref{sec:cons-ebins-slice-1}.
\end{proof}

We also need to know how paths in \(\mathcal{M}\) can be lifted to paths in \(\mathcal{R}\). That this is possible has been shown in \cite{MR0418147}.
Here we need the following local version of this result.

\begin{lemma}
  Let \(\gamma:I\rightarrow \Phi(U\times S)/G\cong S/\Iso_G(g)\) be a path and \((u_0,s_0)\in U\times S\) with \([\Phi(u_0,s_0)]=\gamma(0)\).

  Then there exists a path \(\gamma':I\rightarrow \Phi(U\times S)\), such that \([\gamma'(t)]=\gamma(t)\) for all \(t\in I\) and \(\gamma'(0)=\Phi(u_0,s_0)\).
\end{lemma}
\begin{proof}
  Since \(\Iso_G(g)\) is a compact Lie group, it follows from \cite[Chapter II.6]{MR0413144}, that there is a path \(\gamma'':I\rightarrow S\) such that \(\gamma''(0)=s_0\) and \([\gamma''(t)]=\gamma(t)\) for all \(t\in I\).
  The path \(\gamma'(t)=\Phi(u_0,\gamma''(t))\) then has the required properties.
\end{proof}

\section{{The proofs of Theorems~\ref{sec:introduction-1} and \ref{sec:introduction-2}}}
\label{sec:proof2}

In this section we prove Theorems~\ref{sec:introduction-1} and \ref{sec:introduction-2}. In these proofs we use the same notation as in the previous section.

The first step in the proof is to understand the maps \(\pi_1(\mathcal{M}_{m,k-1})\rightarrow \pi_1(\mathcal{M}_{m,k})\) induced by the inclusions \(\mathcal{M}_{m,k-1}\hookrightarrow \mathcal{M}_{m,k}\), \(k\geq 1\), \(m+k\geq 2\).
This is the content of the next two lemmas.

\begin{lemma}
  For \(k\geq 1\), \(m+k\geq 2\), the map \(\pi_1(\mathcal{M}_{m,k-1})\rightarrow \pi_1(\mathcal{M}_{m,k})\) is surjective.
\end{lemma}
\begin{proof}
  Let \(\gamma: I\rightarrow \mathcal{M}_{m,k}\) with \(\gamma(0)=\gamma(1)\in \mathcal{M}_{m,k-1}\).

Then there are finitely many intervals \([a_j,b_j]\subset ]0,1[\), so that \(\gamma([a_j,b_j])\) is contained in an Ebin chart and \(\gamma^{-1}(\mathcal{M}_{m,k}-\mathcal{M}_{m,k-1})\subset \bigcup_j [a_j,b_j]\).

The path \(\gamma|_{[a_j,b_j]}\) can be lifted to a path in an Ebin slice \(D(E)\). There it can be made transversal to the minimal stratum.
Since the complement of the minimal stratum in an Ebin chart is connected, \(\gamma\) is homotopic to a closed curve in \(\mathcal{M}_{m,k-1}\).
\end{proof}

\begin{lemma}
  For \(k\geq 1\), \(m+k\geq 2\), the kernel of the map \(\pi_1(\mathcal{M}_{m,k-1})\rightarrow \pi_1(\mathcal{M}_{m,k})\) is generated by torsion elements.
\end{lemma}
\begin{proof}
   Let \(\gamma:I\rightarrow \mathcal{M}_{m,k-1}\) be a closed curve and \(\lambda:I\times I\rightarrow \mathcal{M}_{m,k}\) a null homotopy of \(\gamma\).

Then there are finitely many discs \(D_1^2,\dots,D^2_j\subset I\times I-\partial(I\times I)\) with piecewise \(C^1\)-boundary, such that:

\begin{enumerate}
\item \(\lambda^{-1}(\mathcal{M}_{m,k}-\mathcal{M}_{m,k-1})\subset \bigcup_i (D_i^2-\partial D_i^2)\)
\item \(\lambda(D_i^2)\) is contained in an Ebin chart.
\end{enumerate}

Without loss of generality we may assume, that there is a curve \(\sigma:I\rightarrow I\times I\) from the base point to the boundary of \(D_j^2\) such that \(\sigma(t)\not\in \bigcup D_i^2\) for \(t\neq 1\).

Cutting \(I\times I\) along \(\sigma\) leads to a homotopy from \(\gamma\) to \(\lambda(\sigma*\partial D_j^2*\sigma^{-1})\).

Now we can lift \(\lambda(\partial D_j^2)\) to a curve in an Ebin slice and make it transversal to the minimal stratum.
This leads to a homotopy  \(\lambda':I\times I\rightarrow \mathcal{M}_{m,k}\) from \(\gamma\) to \(\lambda(\sigma*\partial (D_j^2)'*\sigma^{-1})\)
with
\begin{enumerate}
\item \(\lambda'(\partial(I\times I))\subset \mathcal{M}_{m,k-1}\)
\item \(\lambda'^{-1}(\mathcal{M}_{m,k}-\mathcal{M}_{m,k-1})\) is contained in \(j-1\) discs as above. 
\end{enumerate}

By induction we see that \(\gamma\) is homotopic in \(\mathcal{M}_{m,k-1}\) to
\begin{equation*}
  \sigma_1*\partial D_1^2*\sigma_1^{-1}*\sigma_2*\partial D_2^2\dots*\sigma_j^{-1}.
\end{equation*}
Therefore it suffices to show that closed curves \(\delta\) which are contained in an Ebin chart are torsion elements.

To do so, lift \(\delta\) to an Ebin slice.
Let \(\delta'\) be the lift of \(\delta\).
Let \(H\) be the compact Lie group which acts on this slice.
Then there is a \(h\in H\), such that \(h\delta'(0)=\delta'(1)\).
Let \(l\) be the order of the class of \(h\) in \(H/H^0\), where \(H^0\) is the identity component of \(H\).

Then
\begin{equation*}
  \delta'*h\delta'*\dots*h^{l-1}\delta'
\end{equation*}
is a lift of \(\delta^l\), which ends and starts in the same component of \(H\delta'(0)\).
Therefore \(\delta^l\) can be lifted to a closed curve in the Ebin slice.
Since the complement of the minimal stratum in the Ebin slice is contractible, it follows that \(\delta\) is torsion.
\end{proof}

As a consequence of these lemmas we get:

\begin{cor}
  \label{sec:the-proof-theorem-4}
  The kernel of \(H_1(\mathcal{M}_{0,1})\rightarrow H_1(\mathcal{M})\) is torsion.
\end{cor}
\begin{proof}
  It follows from the two lemmas above that the kernels of the maps  \(H_1(\mathcal{M}_{m,k-1})\rightarrow H_1(\mathcal{M}_{m,k})\) are torsion.

  To see this note that \(H_1(\mathcal{M}_{m,k})=\pi_1(\mathcal{M}_{m,k})/C_{m,k}\), where \(C_{m,k}\) denotes the commutator subgroup.

 Let \(\gamma\in \pi_1(\mathcal{M}_{m,k-1})\) such that \([\gamma]=0\in H_1(\mathcal{M}_{m,k})\).
 We have to show that \(\gamma\) is contained in a normal subgroup of \(\pi_1(\mathcal{M}_{m,k-1})\) which is generated by torsion elements and \(C_{m,k-1}\).
Since \([\gamma]=0\in H_1(\mathcal{M}_{m,k})\), we have \(\gamma\in C_{m,k}\subset \pi_1(\mathcal{M}_{m,k})\).
Because \(C_{m,k-1}\rightarrow C_{m,k}\) is surjective by the first lemma, it follows from the second lemma that up to torsion elements \(\gamma\) is contained in \(C_{m,k-1}\).
This proves the claim. 

The next step is to show that the kernels of \(H_1(\mathcal{M}_{0,1})\rightarrow H_1(\mathcal{M}_{1,0})\) and \(H_1(\mathcal{M}_{m-1,0})\rightarrow H_1(\mathcal{M}_{m,0})\) are torsion.
This follows from \ref{item:1} of Lemma~\ref{sec:cons-ebins-slice-1}, by taking the limit.

Therefore the statement follows from \ref{item:2} of Lemma~\ref{sec:cons-ebins-slice-1}, by taking the limit.
\end{proof}

By combining Theorem~\ref{sec:the-proof-theorem}, the proof of Theorem~\ref{sec:the-proof-theorem-1}, Lemma \ref{sec:cons-ebins-slice} and Corollary~\ref{sec:the-proof-theorem-4} we now get the following corollary which is Theorem~\ref{sec:introduction-1}.

\begin{cor}
  Let \(n>1\) and \(M\) be a closed spin manifold of dimension \(4n-2\) which is a \(\hat{A}\)-multiplicative fiber in degree \(1\).
Moreover, let \(g_0\in \mathcal{R}^+(M)\).
  Then the image of the map \[\pi_1(\mathcal{R}^+(M),g_0)\rightarrow \pi_1(\mathcal{M}_0^+(M),[g_0])\] contains elements of infinite order.
\end{cor}
\begin{proof}
  Note that \(\mathcal{R}^+(M)\) is an open \(\Diff(M)\)-invariant subset of \(\mathcal{R}(M)\).
  Therefore we can apply the above discussion in the case \(\mathcal{R}=\mathcal{R}^+(M)\).

  We may assume that \(g_0\in \mathcal{R}_{0,1}\).
  Note that, by Lemma~\ref{sec:cons-ebins-slice}, \(\pi_1(\mathcal{R}_{0,1},g_0)\rightarrow \pi_1(\mathcal{R},g_0)\) is an isomorphism.
  Moreover, \(\Diff_0(M)\) acts freely on \(\mathcal{R}_{0,1}\).
  
  Therefore, using Theorem \ref{sec:the-proof-theorem}, one sees as in the proof of Theorem~\ref{sec:the-proof-theorem-1} that there are elements of infinite order in the image of the map \(\pi_1(\mathcal{R}_{0,1},g_0)\rightarrow \pi_1(\mathcal{M}_{0,1},[g_0])\), whose images in \(H_1(\mathcal{M}_{0,1})\) have infinite order.
  
  Hence, the statement follows from Corollary \ref{sec:the-proof-theorem-4}.
\end{proof}

We also have the following theorem, which is Theorem \ref{sec:introduction-2} from the introduction.

\begin{theorem}
Let \(n>2\) be odd and \(0<m_1<m_2<n\) such that \(m_1+m_2=n\).
Moreover, let \(N\) be a closed connected orientable manifold admitting a metric of positive scalar curvature such that
\begin{enumerate}
\item \(H^k(N;\mathbb{Z})=0\) for \(k=m_2-m_1,m_1,m_2,n\),
\item all automorphisms of \(H^*(N;\mathbb{Z})\) are orientation preserving, i.e., induce multiplication with \(1\) in top degree.
\end{enumerate}

  Then, for \(M=S^{m_1}\times S^{m_2}\times S^n\times N\), 
  \[\pi_1(\mathcal{M}_0^+(M),[g_0])\]
  contains elements of infinite order.
  Here \(g_0=h+\epsilon^2 g\) is a product metric with \(h\in \mathcal{R}(S^{m_1}\times S^{m_2}\times S^n)\), \(g\in \mathcal{R}^+(N)\) and \(\epsilon>0\) sufficiently small.
\end{theorem}
\begin{proof}
  Let \(\varphi:S^{m_1}\times S^{m_2}\rightarrow S^n\) be a degree-one map and \(A:S^n\rightarrow SO(n+1)\), such that the image of \([A]\) under the natural map
  \(\pi_n(SO(n+1))\rightarrow \pi_n(S^n)\cong \mathbb{Z}\) is nontrivial.
  We define a diffeomorphism
  \begin{align*}
    f:S^{m_1}\times S^{m_2}\times S^n&\rightarrow S^{m_1}\times S^{m_2}\times S^n&(x,y,z)&\mapsto (x,y,A(\varphi(x,y))z).
  \end{align*}

  Taking the product with the identity on \(N\) we get a diffeomorphism \(F=f\times \Id_N\) of \(M\).
  The map induced by \(F\) on \[H^n(M;\mathbb{Z})\cong H^n(S^{m_1}\times S^{m_2};\mathbb{Z})\oplus H^n(S^n;\mathbb{Z})\]
  is given by
  \begin{align*}
    F^*(\theta_{S^n})&=\theta_{S^n}+a\theta_{S^{m_1}\times S^{m_2}}&F^*(\theta_{S^{m_1}\times S^{m_2}})=\theta_{S^{m_1}\times S^{m_2}},
  \end{align*}
  with \(a\in \mathbb{Z}-\{0\}\).
  Here for an oriented manifold \(M'\), \(\theta_{M'}\) denotes the orientation class of \(M'\).

  We claim that the class of \(F^*\) has infinite order in the abelization of \(\aut_0(H^*(M;\mathbb{Z}))\), the group of orientation preserving automorphisms of \(H^*(M;\mathbb{Z})\).
  To see this note that, by the K\"unneth formula, for \(k=m_1,m_2,n\) we have \(H^k(M;\mathbb{Z})\cong H^k(S^{m_1}\times S^{m_2}\times S^n;\mathbb{Z})\).
  Therefore, every automorphism of \(H^*(M;\mathbb{Z})\) restricts to an automorphism of \(H^*(S^{m_1}\times S^{m_2}\times S^n;\mathbb{Z})\). Furthermore,
  for every orientation-preserving \(\psi\in\aut_0(H^*(M;\mathbb{Z}))\) there is a \(\psi'\in\aut(H^*(N;\mathbb{Z}))\) such that the following diagram commutes.
  \begin{equation*}
    \xymatrix{H^*(M;\mathbb{Z})\ar^{\psi}[r] \ar[d]&H^*(M;\mathbb{Z})\ar[d]\\
      H^*(N;\mathbb{Z})\ar^{\psi'}[r]&H^*(N;\mathbb{Z})}
  \end{equation*}
  Here the vertical maps are induced by the inclusions of the factor \(N\) in \(M\).

  Moreover, we must have \[\psi(\theta_{S^{m_1}\times S^{m_2}})\psi(\theta_{S^n})=\psi(\theta_{S^{m_1}\times S^{m_2}\times S^n})=\theta_{S^{m_1}\times S^{m_2}}\theta_{S^n},\]
  because \(\psi\) and \(\psi'\) are orientation preserving and \(\theta_{M}=\theta_{S^{m_1}\times S^{m_2}\times S^n}\theta_{N}\).

  Therefore it follows that the restriction of \(\psi\) to degree \(n\) is of the form
  \begin{align*}
    \psi(\theta_{S^n})&=\delta(\psi)\theta_{S^n}+a(\psi)\theta_{S^{m_1}\times S^{m_2}}&\psi(\theta_{S^{m_1}\times S^{m_2}})=\delta(\psi)\theta_{S^{m_1}\times S^{m_2}},
  \end{align*}
  with \(\delta(\psi)\in \{\pm 1\}\) and \(a(\psi)\in \mathbb{Z}\).

  Now it is easy to see that the map \(\psi\mapsto \delta(\psi)^{-1}a(\psi)\) defines a group homomorphism \(\aut_0(H^*(M;\mathbb{Z}))\rightarrow \mathbb{Z}\).
  Since \(F^*\) has non-trivial image under this homomorphism it follows that the class of \(F^*\) has infinite order in the abelization of  \(\aut_0(H^*(M;\mathbb{Z}))\).
  In particular, the class of \(F\) has infinite order in the abelization of \(\Diff_0(M)/\Diff_0(M)^0\).

  The next step is to look at the action of \(F\) on the space of metrics of positive scalar curvature on \(M\).

  Since \(F\) has product form, it maps product metrics to product metrics.
  Because for a product metric \(g_0'=h+\epsilon^2g\) we have
  \[\scal_{g_0'}=\scal_h+\frac{1}{\epsilon^2}\scal_g,\]
  the space of product metrics of positive scalar curvature which restrict to a metric of the form  \(\epsilon^2g\) on \(N\) is path connected, if \(g\) has positive scalar curvature.

  Therefore, \(F\) maps the component \(\mathcal{R}'\) of \(g_0\) in \(\mathcal{R}=\mathcal{R}^+(M)\) to itself.
  Since the maximal stratum \(\mathcal{R}_{0,1}\) is \(\Diff(M)\)-invariant, \(F\) also maps \(\mathcal{R}_{0,1}'=\mathcal{R}'\cap \mathcal{R}_{0,1}\) to itself.
  Note that \(\mathcal{R}_{0,1}'\subset \mathcal{R}'\) is open and dense.

  Let \(g_0'\in\mathcal{R}_{0,1}'\) be close to \(g_0\).
  Then there is an isomorphism \(\pi_1(\mathcal{M},[g_0'])\cong\pi_1(\mathcal{M},[g_0])\).
  Moreover, there is an exact sequence 
  \begin{equation*}
    \pi_1(\mathcal{M}_{0,1},[g_0'])\rightarrow \pi_0(\Diff_0(M))\rightarrow \pi_0(\mathcal{R}_{0,1}).
  \end{equation*}

  Since the above map \(\pi_1(\mathcal{M}_{0,1},[g_0'])\rightarrow \pi_0(\Diff_0(M))\) is a group homomorphism, by the above remarks, \([F]\in \Diff_0(M)/\Diff_0(M)^0=\pi_0(\Diff_0(M))\) gives rise to an element \(\alpha\) of infinite order in \(\pi_1(\mathcal{M}_{0,1},[g_0'])\).
  Moreover, the class of \(\alpha\) in \(H_1(\mathcal{M}_{0,1};\mathbb{Z})\) also has  infinite order.
  Hence, the claim follows from Corollary~\ref{sec:the-proof-theorem-4}.
\end{proof}

\small
\bibliography{moduli_non_equi}{}
\bibliographystyle{alpha}

\normalsize
~\\
Michael Wiemeler\\
Mathematisches Institut, Universit\"at M\"unster\\Einsteinstra\ss{}e 62, D-48149 M\"unster, Germany\\
\texttt{wiemelerm@uni-muenster.de}

\end{document}